\documentclass[english,12pt]{article}
\usepackage[utf8]{inputenc}
\pdfoutput=1
\usepackage{amsmath}
\usepackage{tikz}

\usepackage{amsthm}
\usepackage{amssymb}
\usepackage{mathdots}
\usepackage{mathtools}
\usepackage{booktabs}
\usepackage[toc,page]{appendix}
\usepackage{amsxtra}
\usepackage{amsbsy}
\usepackage{verbatim}
\usepackage{comment}
\usepackage{lipsum}
\usepackage{adjustbox}
\usepackage{mathrsfs}
\usepackage{tabularx,ragged2e,booktabs,caption}
\usepackage{leftidx}
\usepackage{longtable}
\usepackage{lipsum}
\usepackage{enumerate}
\usepackage{blindtext}
\usepackage{url}
\usepackage{bbm}
\usepackage{euscript}
\usepackage{pifont}
\RequirePackage{tikz-cd}
\usepackage[T1]{fontenc}
\usepackage{babel}
\usepackage{csquotes}
\usepackage{hyperref}
\usepackage{wasysym}
\usepackage{fancybox}
\usepackage{scalerel,stackengine}
\usepackage{titlesec}
\titleformat{\section}
  {\centering\normalfont\normalsize\bfseries}{\thesection.}{1em}{}
\titleformat{\subsection}
  {\centering\normalfont\normalsize\bfseries}{\thesubsection.}{1em}{}
\hypersetup{
    colorlinks=true,
    citecolor=black,
    linkcolor=blue,
    filecolor=magenta,
    urlcolor=cyan,
}
\stackMath
\newcommand\reallywidehat[1]{%
\savestack{\tmpbox}{\stretchto{%
  \scaleto{%
    \scalerel*[\widthof{\ensuremath{#1}}]{\kern-.6pt\bigwedge\kern-.6pt}%
    {\rule[-\textheight/2]{1ex}{\textheight}}
  }{\textheight}%
}{0.5ex}}%
\stackon[1pt]{#1}{\tmpbox}%
}

\usepackage{graphicx}
\usepackage{epstopdf}

    \usepackage[
    top    = 2.60cm,
    bottom = 2.60cm,
    left   = 2.9 cm,
    right  = 2.9 cm]{geometry}

\newtheorem{thm}{Theorem}[section]
\newtheorem{lem}[thm]{Lemma}

\newtheorem{prop}[thm]{Proposition}

\newtheorem*{conjecture*}{Conjecture}
\newtheorem*{thm*}{Theorem}

\theoremstyle{definition}
\newtheorem{define}[thm]{Definition}
\newtheorem{remark}{Remark}
\newtheorem{example}[thm]{Example}

\newcommand{\1}{\mathbf{1}}

\newcommand{\Z}{\mathbb{Z}}

\newcommand{\K}{\Bbbk}

\newcommand{\BasedAlg}{\operatorname{BasedAlg}}

\newcommand{\rank}{\operatorname{rank}}
\newcommand{\Bim}{\operatorname{Bim}}
\newcommand{\PSU}{\operatorname{PSU}}
\newcommand{\id}{\operatorname{id}}
\newcommand{\Hom}{\operatorname{Hom}}

\newcommand{\ModAlg}{\operatorname{ModAlg}}

\newcommand{\Irr}{\operatorname{Irr}}

\newcommand{\Mod}{\operatorname{Mod}}

\newcommand{\End}{\operatorname{End}}

\newcommand{\KQ}{\operatorname{\K Q}}

\newcommand{\Id}{\operatorname{Id}}
\newcommand{\VecC}{\operatorname{Vec}}

\newcommand{\Fib}{\operatorname{Fib}}
\newcommand{\Rep}{\operatorname{Rep}}

\newcommand{\Fun}{\operatorname{Fun}}

\newcommand{\coRep}{\operatorname{coRep}}

\newcommand{\WHA}[7]
{
    \begin{tikzpicture}[x=0.75pt,y=0.75pt,yscale=-1,xscale=1]

\draw    (223.99,132.68) -- (223.99,177.28) ;
\draw    (223.99,57.69) -- (223.99,102.29) ;
\draw  [draw opacity=0] (224,154.98) .. controls (206.91,154.97) and (193.05,138.18) .. (193.05,117.48) .. controls (193.05,96.78) and (206.91,80) .. (224,79.99) -- (224.02,117.48) -- cycle ; \draw   (224,154.98) .. controls (206.91,154.97) and (193.05,138.18) .. (193.05,117.48) .. controls (193.05,96.78) and (206.91,80) .. (224,79.99) ;  

\draw (177.73,112) node [anchor=north west][inner sep=0.75pt]    {$#7$};
\draw (219.03,180) node [anchor=north west][inner sep=0.75pt]    {$#1$};
\draw (226.25,148) node [anchor=north west][inner sep=0.75pt]    {$#6$};
\draw (220,120) node [anchor=north west][inner sep=0.75pt]  {$#3$};
\draw (220,103.64) node [anchor=north west][inner sep=0.75pt]    {$#2$};
\draw (218.99,42) node [anchor=north west][inner sep=0.75pt]    {$#4$};
\draw (226.25,71.87) node [anchor=north west][inner sep=0.75pt]    {$#5$};

\end{tikzpicture}
}
\usepackage{graphicx} 

\title{Weak Hopf algebra actions as fusion category actions}
\author{Alexander Betz}
\date{}

\begin{document}
\maketitle

\maketitle
\begin{abstract}
This article develops the theory of fusion categories acting on algebras. We will demonstrate that weak Hopf algebra actions on algebras correspond to specific actions of fusion categories. As an application of this theory, we introduce a family of filtered actions of weak Hopf algebras on the path algebra, and for weak Hopf algebras whose representation categories are equivalent to $\PSU(2)_{p-2}$, we describe all of the separable based fusion category actions in terms of weak Hopf algebra actions.
\end{abstract}

\tableofcontents

\section{Introduction}

The objective of this paper is to study quantum symmetries of algebras through the lens of fusion category actions. It's well established that finite groups classify the finite symmetries of an algebra. In the framework of quantum symmetries, finite-dimensional semisimple weak Hopf algebra actions on algebras provide a natural generalization of group actions. Fusion category actions are another realization of quantum symmetries of an algebra, generalizing weak Hopf algebra actions. Fusion categories act on algebra by acting on their category of bimodules, which can be thought of as the non-invertible symmetries of an algebra. We will study these non-invertible symmetries through classifying monoidal functors $F: C\rightarrow \Bim(A)$. Historically, this problem has been studied where $A$ is a $C^*$ algebra \cite{HP17,HHP20,EJ24,CJHP24}. Recently, the algebraic setting where $A$ is assumed to be an associative algebra was explored in \cite{Bet25}, laying the groundwork for a purely algebraic approach to fusion category actions.

Weak Hopf algebras, also known in the literature as quantum groupoids or quantum semigroupoids, are generalizations of Hopf algebras that relax the axioms concerning the unit and counit—specifically, weakening the multiplicativity of the counit and the comultiplicativity of the unit. They were introduced in \cite{BS96} and then further developed in \cite{BNS99, BS00,Szl01,NV02}. It was shown in \cite{Ost03} that the representation category of a semisimple weak Hopf algebra is equivalent to a fusion category. This paper aims to connect actions of fusion categories on algebras and $H$-module algebras for a finite-dimensional semisimple weak Hopf Algebra $H$, building off the results in \cite{Bet25}. We produce the following theorem connecting these two types of actions,

\begin{thm}
    Let $H$ be a semisimple weak Hopf algebra. Equivalence classes of based actions of fusion categories of $\coRep(H)$ on $A$ over $F_H$ are in bijection with left $H$ actions on $A$ up to $J$-twisted isomorphism of $H$.
\end{thm}

We apply this theoretical framework to classifying quantum symmetries of the path algebra through the lens of weak Hopf algebra actions. It's a classic result that all finite-dimensional algebras are Morita equivalent to a quotient of a path algebra, thus motivating our interest in studying the symmetries of path algebras. In particular, \cite{EKW21} studied actions of finite-dimensional Hopf algebras on path algebras preserving the canonical grading on $\KQ$. Building on this, \cite{Bet25} extended these results to graded preserving actions of fusion categories on path algebras. For the family of fusion categories $\PSU(2)_{p-2}$, it was shown that all separable based actions were graded actions. From the perspective of weak Hopf algebras, there have been connections established between
face algebras and their (co)actions on quivers \cite{HWWW,CW23, Hay96}. There has also been some progress in studying weak Hopf algebra actions on algebras \cite{CHWW,Szl00,NSW98}. We will use these results as motivation to classify some weak Hopf algebra actions on the path algebra.

\begin{thm}
    There is a family of graded actions of weak Hopf algebras up to $J-$ twisted algebra isomorphism on path algebras classified by the following data,
     \begin{enumerate}
        \item A semisimple $\coRep(H)$ module category structure on $M$ up to invertible $\coRep(H)$ module functors of the form $(\Id,\rho)$.
        \item A conjugacy class of objects $[(Q,\phi_{-})]$ in $\End_{\coRep(H)}(M)$ up to conjugation by 
 a $\coRep(H)$ module endofunctors of the form $(\Id,\rho)$. 
   \end{enumerate} 
\end{thm}

We observe that for graded separable idempotent split based actions of fusion categories on $\KQ$ defined in \cite{Bet25}, are actually weak Hopf algebra actions on $\KQ$ and then apply the full classification of $\PSU(2)_{p-2}$ action on path algebras and explain those actions in terms of weak Hopf algebra data.

\begin{thm}
  Let $H$ be a weak Hopf algebra generated by $\PSU(2)_{p-2}$ and a semisimple module category of $\PSU(2)_{p-2}.$ Then all actions of $H$ on $\KQ$ are graded actions.
\end{thm}

\section{\texorpdfstring{\centering Preliminaries}{Preliminaries}} 

This section contains the necessary material about tensor categories, and weak Hopf algebras needed for the rest of the paper. 

\subsection{Tensor Categories}

For the rest of the paper, $X, Y,Z$ are objects in a category. An object is simple in a linear category if $\Hom(X, X)\cong \K$. A monoidal category is a quintuple $(C,\otimes,a,\1,i)$ where $C$ is a category, $\otimes:C\times C\rightarrow C$ is a bifunctor, $a:(X\otimes Y)\otimes Z\rightarrow X\otimes (Y\otimes Z)$ is a natural isomorphism, $\1$ is the unit object and $i:\1\otimes \1\rightarrow \1$ is an isomorphism satisfying some categorical coherences. A monoidal category is said to be rigid if for an object $X$ there exists a left and right dual satisfying the classical duality relations. 

A monoidal functor from $(C,\otimes,a,\1,i)$ to $(C',\otimes',a',\1',i')$ is a pair $(F,J_{-,-})$ where $F:C\rightarrow C'$ is a functor and $J_{-,-}$ is a natural isomorphism between $F(X\otimes Y)$ and $F(X)\otimes F(Y)$ satisfying some coherences. A monoidal functor is called a lax monoidal functor if $J_{-,-}$ is a natural transformation. A fusion category is a finite $\K$ linear rigid abelian semisimple monoidal category such that the unit object is simple. Examples of fusion categories are $\Rep(H)$ where $H$ is a semisimple Hopf algebra and $\Fib$ is the Fibonacci Category. 
\begin{example}
    A non-trivial example of a family of fusion categories is $\PSU(2)_{p-2}$. These categories have simple objects indexed by $\{X_{2j},0\leq j\leq \frac{p-3}{2}\}$ where $X_0\cong \1$. These simple objects satify the fusion rules,
\[X_{2j}\otimes X_{2i}=\sum N_{2i,2j}^{2m}V_{2m}.\]

Where the fusion coefficient $N_{2i,2j}^{2m}$ is,

\[N_{2i,2j}^{2m}=
\begin{cases}
    1 \hspace{1mm}\text{if} \hspace{1mm} |2i-2j|\leq 2m\leq \min \{2i+2j, 2(p-2)-2i-2j \} \\
    0 \hspace{1mm} \text{otherwise}
    
\end{cases}.\]

The dimension of a simple object $X_{2j}$, is the quantum integer $[2j+1]_q$ where $q$ is a $2^{pth}$ root of unity. Since this number is irrational except when $j=0$ it follows that these categories cannot be monoidally equivalent to $\Rep(H)$ for a semisimple Hopf algebra $H$. The Fibonacci Category, $\Fib$, is a the first non-trivial fusion category in this family $\Fib \cong \PSU(2)_{5-2}$. By the fusion rules above, $\PSU(2)_3$ has two simple elements $1$ and $\tau$ satisfying $\tau \otimes \tau \cong 1\oplus \tau$.
\end{example}

It's a well-known result that all fusion categories are equivalent to the (co)representation categories of a weak Hopf Algebra \cite{EGNO16}. A left module category over a monoidal category $C$ is a category $M$ equipped with a bifunctor $\triangleright:C\times M\rightarrow M$ and a natural isomorphism $m_{X,Y,Z}:X\triangleright (Y\triangleright Z)\rightarrow (X\otimes Y)\triangleright Z$ satisfying some coherence. In \cite[Chapter 7]{EGNO16} it is shown that module category structures on $M$ are in bijection with monoidal functors $F:C\rightarrow \End(M)$. There is an important theorem by Eilenberg and Watts \cite{Eil60,Wat} that states if  $R$ and $S$ are rings, then 
\[\Bim(R,S) \cong \Fun_{coc}(\Mod_R,\Mod_S).\]
Where $\Fun_{coc}$ is the category of colimit-preserving additive functors. This theorem is essential for our work since it states that a module category structure on $\Mod(A)$ is equivalent to a monoidal functor $F:C\rightarrow \Bim(A)$. Given two $C$ module categories $M$ and $N$ we can define a $C$ module functor $(G,\eta_{-,-})$ where $G:N\rightarrow M$ is a functor and $\eta_{X,M}:X\triangleright G(M)\rightarrow G(X\triangleright M)$ is a natural isomorphism. The category of $C$ module endofunctors $\End_C(M)$ is a monoidal category. If $C$ is a fusion category and $M$ is a semisimple module category, then $\End_C(M)$ is a multifusion category. If, in addition, $M$ is indecomposable, then $\End_C(M)$ is a fusion category. For a fusion category $C$, an algebra $A$ internal to $C$ is a triple $(A,m,i)$ where $A$ is an object in $C$, and $m:A\otimes A$, $i:\1\rightarrow A$ are multiplication and the unit morphisms satisfying the usual algebra coherences. A right module $X$ over $A$ is a pair $(X,\rho)$ where $X$ is an object in $C$ and $\rho:X\otimes A \rightarrow A$ is a morphism satisfying the usual module coherences. There is a natural left $C$-Module structure on the category of right $A$ modules $\Mod(A)$ for some internal algebra $A$ using the  associator and the action of $A$ on modules. It was shown in \cite{Ost03} that all semisimple module categories in a fusion category can be understood in terms of categories of modules over an algebra internal to $A$. 
If $M \cong \Mod(A)$ is a left $C$-Module category, then the category of $C$ module endofunctors $\End_C(M) \cong \Bim(A)^{op}$ the opposite tensor product category \protect\cite{EGNO16}.

\subsection{Weak Hopf Algebras}

In this subsection, we will introduce the basics of weak Hopf algebras. We will use Sweedlers notation for comultiplication $\Delta(x)=x_{(1)}\otimes x_{(2)}$. For the rest of this paper, $H$ will denote a finite-dimensional semisimple weak Hopf algebra.
\begin{define}
    A weak Hopf algebra $H$ is a sixtuple $(H,m,u,\Delta,\epsilon,S)$ where $(H,m,u)$ is an algebra, $(H,\delta,\epsilon)$ is a coalgebra and $S$ is a linear map called the antipode satisfying,

    \begin{enumerate}
        \item $\Delta(gh)=\Delta(g)\Delta(h)$.
        \item $(\Delta\otimes \id)\Delta(1)=(\Delta(1)\otimes 1)(1\otimes \Delta(1))=(1\otimes \Delta(1))(\Delta(1)\otimes 1).$
        \item $\epsilon(ghf)=\epsilon(fg_{(1)})\epsilon(g_{(2)}h)=\epsilon(fg_{(2)})\epsilon(g_{(1)}h).$

        \item $m(\id\otimes S)\Delta(h)=(\epsilon \otimes \id)(\delta(1))(h\otimes 1).$
        \item $m(S\otimes \id)\Delta(h)=(\id \otimes \epsilon)(1\otimes h)(\Delta(1)).$
        \item $S(h)=S(h_1)h_2S(h_3).$
    \end{enumerate}
\end{define}

The main difference between a Hopf algebra and a weak Hopf algebra is the weakening of the multiplication of the counit and the weakening of the comultiplication of the unit. If a weak Hopf algebra has a unit-preserving coproduct, or a product-preserving counit, then it is in fact a Hopf algebra. The easiest example of a weak Hopf algebra is the direct sum of two Hopf algebras. 
There are two special maps for a weak Hopf algebra, denoted the target and source maps. They are defined as follows,

\[ \epsilon_t(h)=\epsilon(1_{(1)}h)1_{(2)},\hspace{1mm} \epsilon_s(h)=1_{(1)}\epsilon(h1_{(2)}).\]

The images of these maps define separable algebras called the target and source algebras, respectively, denoted $H^t$ and $H^s$. As an algebra, $H^t$ is generated by the right tensor factors of the shortest possible presentation of $1_{(2)}$. $H^s$ is the same but uses left tensor factors of $\Delta(1)$. It has been shown that $\Rep(H)$ and $\coRep(H)$ are fusion categories when $H$ is a finite-dimensional semisimple weak Hopf algebra \cite{EGNO16}. Furthermore, it has been shown that all fusion categories are monoidally equivalent to the category of representations of a weak Hopf algebra \cite{Szl01}. The tensor product of (co)representation is taken over the target algebra $\otimes_{H^t}$ instead of $\otimes_{\K}$ for Hopf algebras.

\begin{example}[Direct sum of Group Algebra]
Most weak Hopf algebras are challenging to write down. Let $H\cong \K G_1\oplus \K G_2$. This is an example of a weak Hopf algebra that is not a Hopf algebra. The maps are defined as follows,

\[m(g_1+g_2,h_1+h_2)=g_1h_1+g_2h_2.\]
\[\Delta(g_1+g_2)=g_1\otimes g_1+g_2\otimes g_2.\]
\[\epsilon(g_1+g_2)=\epsilon(g_1)+\epsilon(g_2).\]
\[1=1_{g_1}+1_{g_2}.\]
\[S(g_1+g_2)=g^{-1}_1+g^{-1}_2.\]

This fails to be a Hopf algebra since the coproduct of the unit is ${1_{g_1}}\otimes 1_{g_1}+1_{g_2}\otimes 1_{g_2}$, which is not equal to $(1_{g_1}+1_{g_2})\otimes(1_{g_1}+1_{g_2})$. Thus, the comultiplication is not unital, and we conclude that $\K G_1\oplus \K G_2$ is a weak Hopf algebra and not a Hopf algebra.
\end{example}

Given a semisimple weak Hopf Algebra $(H,m,u,\Delta,\epsilon,S)$ we can define three other weak Hopf algebras

\[H^{op}=(H,m^{op},u,\Delta,\epsilon,S^{-1}).\]
\[H_{cop}=(H,m,u,\Delta^{op},\epsilon,S^{-1}).\]
\[H_{cop}^{op}=(H,m^{op},u,\Delta^{op},\epsilon,S).\]

$H_{cop}^{op}$ is canonically isomorphic to $H$.

\begin{define}
    An algebra $A$ is a left $H$ module algebra if,

    \begin{enumerate}
        \item $A$ is an $H$ module
        \item $h(ab)=h_{(1)}(a) h_{(2)}(b)$
        \item $h( 1_A)=\epsilon_t(h)(1_A)$
    \end{enumerate}
\end{define}

\begin{prop}
    Given an $H$ module algebra structure on $A$ then $1_{(1)}$ and $1_{(2)}$ act trivially on any $a \in A$.
\end{prop}

\begin{proof}
    $1_H(1_A)=1_A$ but using the comultiplication $1_A=1_{(1)}(1_A)1_{(2)}(1_A)$. Notice that $1_{(1)}(1_A)$ is equal to $\epsilon_t(1_{(1)})1_A=\epsilon(1_{(1)})1_{(2)}(1_A)=1_H1_A$ thus $1_{(1)}(1_A)=1_A$ which implies that $1_{(2)}(1_A)=1_A$. Now for any $a\in A$ we have that $a=a\cdot 1$ and $1_H(a)=a$. Then applying the comultiplication property it follows that $1_H(a)=1_{(1)}(1_A)1_{(2)}(a)=1_{(1)}(a)1_{(2)}(1_A)=a$ which implies that $1_{(1)}$ and $1_{(2)}$ act trivially on any $a \in A$.
\end{proof}

$H$ module algebras are algebra objects internal to the category $\Rep(H)$. As mentioned earlier, the category of right modules over an internal algebra object is a left $\Rep(H)$ module category.

A relevant example of a weak Hopf algebra is the weak Hopf algebra associated with a $C$ module category. Given a fusion category $C$ and a left semisimple $C$ module category $M$, and a progenerator $m$, we can produce a weak Hopf algebra using this data. A natural choice of progenerator $m$ is the sum of all simple objects in $M$. When $M\cong C$, the corresponding weak Hopf algebra is the face algebra associated to a fusion category \cite{Hay96}. In \cite{BLV}, these weak Hopf algebra constructions are written down in terms of a picture basis that we will use later in the paper.

\section{Weak Hopf Algebra Actions and Based Actions}

This is the main section of the paper where we will establish a connection between fusion category actions and finite-dimensional semisimple weak Hopf algebra actions. This builds off the connection established in \cite{Bet25} between $\coRep(H)$ actions and $H$ actions for a semisimple Hopf algebra. An action of a fusion category $C$ on an algebra is defined as a $C$ module structure on $\Mod(A)$. By the Eilenberg-Watts Theorem \cite{Eil60,Wat}, this is equivalent to a monoidal functor $F:C\rightarrow \Bim(A)$. We want to study fusion category actions on algebras with some more structure, thus we define a based action of a fusion category on an algebra \cite{Bet25}. 

\begin{define}
    A based action of a fusion category $C$ on an algebra $A$ is a triple $(A,F,V_{-})$ where $(F,J_{-,-}):C\rightarrow \Bim(A)$ is a monoidal functor, and $V_X\subset F(X)$ is a finite dimensional subspace satifying that $\1_A\in V_1$, and morphisms preserve these subspaces.
\end{define}

There is a natural notion of equivalence of based actions that we will use in this paper, which is defined as follows.

\begin{define}
    Two based actions $(F,J,V)$ and $(G,H,U)$ on $A$ are equivalent if there is a monoidal natural isomorphism of based actions between them.
\end{define}

There are a few properties that we want to use for based actions; we will list them now. Refer to \cite{Bet25} for a comprehensive introduction to based actions.

\begin{lem}[\cite{Bet25} Lemma 3.9]
    Let $(F,J,V)$ be a based action of a fusion category $C$ on $A$. There is an induced lax monoidal functor $G:C\rightarrow \VecC_{f.d}$ where $G(X)\cong V_X$.
\end{lem}

 It was shown in \cite{Bet25} that based actions fully generalize actions of a semisimple Hopf algebra. In this section, we will show that this result generalizes to weak Hopf algebra actions on algebras. We will first define some other important definitions about based actions.

  \begin{define}[\cite{Bet25}]
 Based actions of $C$ on $A$ over a functor $G$ are based actions such that the induced lax monoidal functor is monoidally naturally isomorphic to $G$. 
\end{define}

\begin{define}[\cite{Bet25}]
    A based action is separable if $V_1$ is a semisimple algebra in $\VecC_{f.d.}$.
\end{define}

Using this definition, we will now have the data necessary to equate weak Hopf algebra actions on $A$ and certain based actions of $\coRep(H)$ on $A$ for a semisimple weak Hopf algebra $H$. We will first show how a $H$ module algebra produces a monoidal functor $F:\coRep(H)\rightarrow \Bim(A)$. Notice that if there is a left $H$ action on $A$, it follows that there is a left $H^t$ action on $A$. We can define a right action of $H^t$ on $A$ using the antipode $S$ to state that $A$ is an $H^t$ bimodule.

\begin{lem}\label{lem:coRep becomes bim}
 Let $A$ be a left $\Rep(H)$ module algebra and let $(V,\rho)$ be a $H$-comdule, then $A\otimes_{H^t} V$ is an $A-A$ bimodule.
\end{lem}

\begin{proof}
    Let $V$ be a corepresentation of $H$. Then we can define an $A$ bimodule structure on $A \otimes_{H^t} V$. The left action of $A$ on $A \otimes V$ is left multiplication in $A$. The right action of $A$ on $A\otimes_{H^t} V$ is done by applying the comodule action on $V$. That is $(a\otimes_{H^t}v)a':=ah_{(1)}(a')\otimes_{H^t}v_{(2)}$.
\end{proof}

\begin{lem}\label{lem:Hopf action gives based action}
    Let $A$ be a left $H$ module algebra. Then this action induces a monoidal functor $F:\coRep(H)\rightarrow \Bim(A)$ defined by $F((X,\rho))\cong A\otimes_{H^t} X$.
\end{lem}
\begin{proof}
    First, we shall show that $F$ is a functor. Given a morphism $f:X\rightarrow Y$ then $F(f)=\id_A\otimes_{H^t}f$. This is a left $A$ module morphism and will be a right $A$ module morphism because $f$ is a comodule homomorphism. $F$ respects the composition of morphisms and sends the identity to the identity, therefore $F$ is a functor.

    Now we will show that $F$ is monoidal. We need to construct a natural isomorphism $J_{X,Y}:F(X)\otimes_A F(Y) \rightarrow F(X\otimes Y)$.  Let, 
    
    \[J_{X,Y}:=\alpha_{A,X,Y}\circ(r_A\circ \id_Y)\circ \alpha^{-1}_{A\otimes_{H^t}X,A,Y}.\]

    $J_{-,-}$ is a natural isomorphism because of the naturality of each component. We manually check that this pair ($F,J)$ satisfies the necessary criterion to be a monoidal functor, and the result follows.
    \end{proof}

We see that the above lemma implies that $\coRep(H)$ actions generalize weak Hopf algebra actions of $H$ on $A$.  It was shown in \cite{Bet25} that for semisimple Hopf algebras, one could go back using the based actions over the canonical fiber functor. Weak Hopf algebras don't have a canonical fiber functor; instead, they have a weak fiber functor coming from the composite functor coming from the functor into bimodules composed with the forgetful functor from bimodules to vector spaces. We will denote this weak fiber functor $F_H$. Up to the equivalence of module categories, we can understand $F_H$ to factor through bimodules of a commutative separable algebra. In the case where $F_H$ factors through a semisimple commutative algebra, a pictorial representation of $F_H$ is as follows: first, we can write down the $C$ module category as the following picture diagram. This produces a monoidal functor from $C$ into $\Bim(\K^{\rank(M)})$. $X$ is an element in the fusion category and $m$ is a progenerator for the semisimple module category.

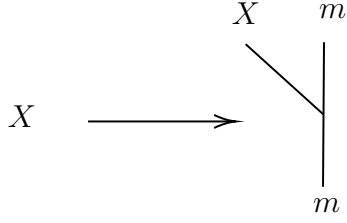
\begin{center}
\tikzset{every picture/.style={line width=0.75pt}} 

\tikzset{every picture/.style={line width=0.75pt}} 

\begin{tikzpicture}[x=0.75pt,y=0.75pt,yscale=-1,xscale=1]

\draw    (348,87) -- (347.5,160) ;
\draw    (347.75,123.5) -- (308.5,88) ;
\draw    (229,127) -- (301.5,127) ;
\draw [shift={(303.5,127)}, rotate = 180] [color={rgb, 255:red, 0; green, 0; blue, 0 }  ][line width=0.75]    (10.93,-3.29) .. controls (6.95,-1.4) and (3.31,-0.3) .. (0,0) .. controls (3.31,0.3) and (6.95,1.4) .. (10.93,3.29)   ;

\draw (341,164.4) node [anchor=north west][inner sep=0.75pt]    {$m$};
\draw (344,66.4) node [anchor=north west][inner sep=0.75pt]    {$m$};
\draw (300,65.4) node [anchor=north west][inner sep=0.75pt]    {$X$};
\draw (187,117.4) node [anchor=north west][inner sep=0.75pt]    {$X$};

\end{tikzpicture}

\end{center}

There is an algebra structure on $\End(m)=\K^{\rank(M)}$ and these pictures are left and right $\End(m)$ bimodules. The action on these bimodules is just by stacking morphisms in $\End(M)$ on the picture above. Then the weak forgetful functor $F_H$ is the monoidal functor where $V_1\cong\End(m)$ and $V_X$ are the bimodules described above, using the pictures composed with the forgetful functor to vector spaces. When $M\cong C$, we get the well-known face algebra associated to a fusion category \cite{Hay96,Hay99}. We formally define the weak fiber functor as follows,

\begin{define}
    Let $C\cong \coRep(H)$, then $F_H:\coRep(H)\rightarrow \VecC_{f.d}$ is the composite functor $F:C\rightarrow \Bim(V_1)$ composed with the forgetful functor $H:\Bim(V_1)\rightarrow \VecC_{f.d}$.
\end{define}

This forgetful functor is the reconstruction functor that, given a semisimple $C$ module category $M$, we can construct a weak Hopf algebra $H$ such that $C\cong \Rep(H)$ and $M\cong \Rep(H^t)$.

Now we state and prove some lemmas connecting based actions of $\coRep(H)$ on $A$ over $F_H$ and $H$ actions on $A$.

\begin{lem}
    Given a left action of a weak Hopf algebra $H$ on $A$, the induced monoidal functor $F:\coRep(H)\rightarrow \Bim(A)$ is a based action over $F_H$.
\end{lem}
\begin{proof}
   This follows immediately because of how morphisms and the associator are defined.
\end{proof}

Now, since $H$ is a comodule algebra internal to $\coRep(H)$, there is a slight technicality here, as comodules are morphisms $\rho:M\rightarrow M\otimes H$, which is in contrast to the tensor product being over the target algebra in $\coRep(H)$. It was proven generally in \cite{BCM02} and written down in terms of weak bialgebras in \cite{WWW22} that the category of internal $H$ comodule algebras and the category of $H$ module coalgebras are isomorphic as categories, where the isomorphism is the forgetful functor. With this in mind, we understand enough to produce a lemma going back from special based actions of $C$ on $A$ to actions of a weak Hopf algebra on $A$.

\begin{lem} \label{lem:based action give Hopf action}
Let $H$ be a semisimple weak Hopf algebra. Then a based action of $\coRep(H)$ on $A$ over the weak fiber functor $F_H$ induces an $H$ module algebra structure on $A$.
\end{lem}

\begin{proof}
  Let $F:\coRep(H)\rightarrow \Bim(A)$ be a based action of $\coRep(H)$ over $F_H$ on $A$. This implies that $F_H\cong G:\coRep(H)\rightarrow \VecC$.  Let $h\in H$ and $a\in A$, by our half braiding $ha=\sum a_ih_i$.
  
  We can define our action of $H$ on $A$ by using comultiplication and the half braiding. In particular $ha=h_{(1)}(a)h_{(2)}$. Notice that since our base spaces span left/right projective basis, it follows that $ah=\sum_ik^ib^i=\sum_ik^i_{(1)}(b^i)k^i_{(2)}$ for some $h,k\in H$. It follows that multiplying by $\Delta(1)$ on these two terms doesn't change the product. This implies that our bimodule is of the form $A\otimes_{H^t}H$. We will show that this gives a left $H$ module algebra structure on $A$. We will first consider, 

  \[h ka=hk_{(1)}(a)k_{(2)})=(h_{(1)})(k_{(1)}(a))h_{(2)} k_{(2)}=(h_{(1)}k_{(1)})(a)h_{(2)}k_{(2)}\]

  \[h ka=(h k)_{(1)}(a)(h k)_{(2)}=(h_{(1)}k_{(1)})(a)h_{(2)} k_{(2)}.\]

  These are equal because our multiplication is a coalgebra homomorphism. Also, since $1_H=1_A$ it follows that $a1_H=a=1_Ha$, by our definition of our action, this implies the braiding is $1_{(1)}(a)1_{(2)}=a1_H$. Which implies that this is equivalent to $\Delta(1)(1_H(a)1_H)$, thus $1_H(a)=a$.
  
  We shall now show that $A$ has the $H$ module algebra structure. Consider,

  \[h(ab)=h_{(1)}(ab)h_{(2)}\]

  \[(ha)b=h_{(1)}(a)h_{(2)}b=h_{(1)}(a)h_{(21)}(b)h_{(22)}.\]

 By our bimodule half braiding, these are equal, and the comultiplication is coassociative. This implies the action is compatible with comultiplication in $H$ and multiplication in $A$.
 
 Finally, note that, 
 
 \[h1_A=1_Ah\]
 By the definition of a weak Hopf algebra, we have a special map $\epsilon_t$.
 \[\epsilon_t(h_{(1)})=\epsilon(1_{(1)}h_{(1)})1_{(2)}=\epsilon(h_{(1)})1_{(2)}.\]
 Which implies that,
$h1_A=h_{(1)}1_Ah_{(2)}$.  Consider $\epsilon_t(h_{(1)})(1_A)=\epsilon(h_{(1)})1_{(2)}(1_A)$. Then, since we have shown that the action is compatible with comultiplication, it follows that $1_{(2)}(1_A)=1_A$, which is a property of $H$ module algebras. Thus, $A$ has a left $H$ module algebra structure. 
\end{proof}

Thus, for a semisimple weak Hopf algebra $H$, $H$ action on $A$ can be understood in terms of $\coRep(H)$ actions over $F_{H}$ and vice versa. We will now explain what classifies a based action over $F_H$ up to monoidal natural isomorphism. We need to define a special type of algebra isomorphism between weak Hopf algebras.

\begin{define}
Given two semisimple weak Hopf algebras $H_1$ and $H_2$, a $J$-twisted algebra isomorphism is an isomorphism $\alpha_{H}:H_1\rightarrow H_2$ that is a component morphism in a monoidal natural automorphism of the functor $F:C\rightarrow \Bim(H^t)$ such that,

\[m \circ J^2_{H_2,H_2}\circ \alpha_{H_1}\otimes \alpha_{H_1}=\alpha_{H_1}\circ m\circ J^1_{H_1,H_1},\]
\end{define}

\begin{thm}\label{thm:WHA action mod cat}
    Let $H$ be a semisimple weak Hopf algebra. Equivalence classes of based actions of fusion categories of $\coRep(H)$ on $A$ over $F_H$ are in bijection with left $H$ actions on $A$ up to $J$-twisted isomorphism of $H$.
\end{thm}

\begin{proof}
Let's start with two monoidally naturally isomorphic based actions of  $\coRep(H)$ on $A$ over $F_H$ with monoidal natural isomorphism $\alpha$. That is, there are two monoidal functors $F_i:\coRep(H_i)\rightarrow \Bim(A)$ that are monoidally naturally isomorphic. This implies that there is a monoidal natural isomorphism that preserves the base spaces, i.e, $\alpha:V_X\rightarrow U_X$. This implies that $\alpha$ sends $H_1$ to $H_2$. Since multiplication is a comodule homomorphism and $\alpha$ is a monoidal natural isomorphism, it follows that,
\[m \circ J^2_{H_2,H_2}\circ \alpha_{H_1}\otimes \alpha_{H_1}=\alpha_{H_1}\circ m\circ J^1_{H_1,H_1}.\]
This implies that $\alpha$ is an $J$-twisted algebra isomorphism.

Conversely, if we consider to $J$-twisted algebra isomorphic to weak Hopf algebras $H_1$ and $H_2$, then this defines a natural isomorphism of $F:C\rightarrow \Bim(H^t)$ and itself. These two weak Hopf algebra actions produce based actions over $F_H$. Since these $F_H$ functors are naturally isomorphic, it follows that the corresponding based actions will be monoidally naturally isomorphic as based actions.
\end{proof}

Now we can define an equivalence category of based actions over $F_H$ and the corresponding category of weak Hopf algebra actions up to strong isomorphism. We define the following two categories,

\begin{define}
    Define the $\BasedAlg(C,F_H)$ category of based actions of $C$ over $F_H$ where objects are triples $(F,J,V)$ up to monoidal natural isomorphism and morphisms are left f.g projective $A-B$ based bimodules where the set is $B$ and the base space is $1_B$ and the isomorphism between $F(X)\otimes_A M\cong M\otimes_B G(X)$ restricts to the identity on the base spaces.
\end{define}

\begin{define}
    Define the category $\ModAlg(H)$ where objects are weak Hopf algebra actions of $H$ on $A$ up to $J$-twisted algebra isomorphism of $H$ and morphisms are algebra homomorphisms $f:A\rightarrow B$ that are also compatible with the $H$ action 
\end{define}

\begin{lem}
    There is an equivalence of categories between  $\BasedAlg(\coRep(H),F_H)$ and $\ModAlg(H)$
\end{lem}

\begin{proof}
    We already showed that since there is a bijection between these sets, our functor that takes a $J$-twisted algebra isomorphism class of weak Hopf algebra and assigns it a based action up to monoidal natural isomorphism. Thus, this map will be essentially surjective. Therefore, it suffices to show that this functor is bijective on homomorphism spaces. Consider $\Hom_H(B,A)$ the space of algebra homomorphisms that are also $H$ module homomorphisms. For each morphism, we can generate a based bimodule where the set is $B$ and the base space is $1_A$. The right action is action by $f(b)$ and the left $A$ action is acting by $a$. Notice that this creates a based bimodule $A-B$ $\leftidx_Af_B$ bimodule $A$ over $1_A$. This process is injective because the bimodule and the half braiding are entirely determined by $f$. Now consider a left dualizable $A-B$ based bimodule where the set is $A$ and the base space is $1_A$. This implies that if  $(F,J,V)$ and $(G,K,U)$ are based actions on $A$ and $B$ respectively such that $F(X)\otimes_A \leftidx_Af_B\cong G(X)\otimes_B \leftidx_Af_B$ and that this restricts to the based spaces $U_X$ and $V_X$. Thus, consider $1_A\otimes_B U_X$ and the action of $b$ on the right. The first half of the braiding yields
    \[(1_A\otimes_B h)b=1_Ah_{(1)}(b)\otimes_B h_{(2)}=f(h_{(1)}(b))1_A\otimes_B h_{(2)}.\]
     Conversely, multiplying on the right by b to $V_X\otimes_A 1_A$ yields

    \[(h\otimes_A 1_A)b=hf(x)\otimes_A 1_A)=h_{(1)}(f(b))h_{(2)}\otimes_A 1_A\]

    If these are indeed compatible, then the action of $H$ on $A$ and $B$ must be intertwined with the algebra morphism $f$. Thus giving us an $H$ module structure on $f$. Therefore, the map is surjective, and the result follows since there is a functor that is fully faithful and essentially surjective.
\end{proof}

With our understanding of weak Hopf algebra actions and $\coRep(H)$ actions over the weak fiber functor, notice that if an action exists, then there is an embedding of $H^t$ inside $A$. Since $H^t$ has non-trivial idempotents, it follows that $A$ must have non-trivial idempotents, leading us to the following lemma.

\begin{lem}
    If $A$ is an algebra and there are no non-trivial idempotents, then there is no action of a weak Hopf algebra on $A$.
\end{lem}

\begin{proof}
    If $A$ has only $0$ and $1$ as idempotents, then it follows that $H^t$ is not a subalgebra of $A$ and has non-trivial idempotents; then so must $A$.
\end{proof}

\begin{example}
If $A$ is an integral domain then $x^2=x$ then $x^2-x=0$ which implies that $x(x-1)=0$. It follows that $x=0$ or $x=1$; thus, the only idempotents of a domain are $0$ and $1$, and therefore, there is no weak Hopf algebra action of $H$ on an integral domain by the lemma above.
\end{example}

\section{Weak Hopf Algebra Actions on Path Algebras}

In this section we want to construct actions of weak Hopf algebras on path algebras. As a reminder, a quiver $Q=(V,E,s,t)$ is comprised of a vertex set $V$, edge set $E$, a target map $t:E\rightarrow V$, and a source map $s:E\rightarrow V$. Given a quiver $Q$, we define the path algebra over that quiver $Q$, denoted $\KQ$.

\begin{define}[Path Algebra]
    Let $Q$ be a Quiver. Let $\K$ be a field. The Path Algebra $\KQ$ is a graded $\K$ algebra whose basis is the set of all paths of length $l\geq 0$ in $Q$. The product of paths is defined as
    $$e_{i_1}...e_{i_n}*e_{j_1}....e_{j_m}= \delta_{s(e_{i_n})t(e_{j_1})}e_{i_1}...e_{i_n}e_{j_1}....e_{j_m}$$
\end{define}

This is an associative algebra, and if the vertex set is finite, then $\KQ$ is unital. In addition, the vertices form a complete set of primitive orthogonal idempotents in the path algebra such that all other complete sets of primitive orthogonal idempotents are conjugate to the vertices. A path algebra is finite-dimensional if and only if it doesn't contain a cycle.

We begin by stating an important lemma about weak Hopf algebra actions and idempotent split actions of fusion categories on path algebras. This allows us to build families of actions of weak Hopf algebras on $\KQ$ from the fusion category actions created in \cite{Bet25}. An action of a weak Hopf algebra on $\KQ$ corresponds to a separable action of $\coRep(H)$ on $\KQ$ over $F_H$. It was shown in \cite{Bet25} that a separable subalgebra of $\KQ$ is isomorphic to $\K^n$ and is conjugate to when the projections $\KQ$ are sums of vertex projections. A separable idempotent split based action is when $V_1=\K^{|\KQ_0|}$, that is, when the $V_1$ space consists of the minimal orthogonal vertex projections of $\KQ$.

\begin{lem}\label{lem:idem split FH}
    Let $C$ be a fusion category, then every equivalence class of graded separable idempotent split based actions of $C\cong \coRep(H)$ up to conjugacy corresponds to a weak Hopf algebra action up to $J$-twisted algebra isomorphism.
\end{lem}

\begin{proof}
   Given a graded idempotent split based action of $C$ on $\KQ$ corresponds to a monoidal functor $F:C\rightarrow \Bim(\KQ)$ such that $V_1$ is the canonical set of vertex projections. Since this is a based action, it follows that $J_{X,Y}:F(X)\otimes_{\KQ} F(Y)$ is a natural isomorphism of based bimodules, implying that the image of $V_X\otimes V_Y$ under the $\KQ$ balanced map is $V_{X\otimes Y}$. That is $V_X\otimes_{\KQ} V_Y$ is linearly isomorphic to $V_{X\otimes Y}$. Now, to be a weak Hopf algebra action, there is a monoidal functor from $C$ into $\Bim(V_1)$. For a based action, this functor is generally lax monoidal. Thus, it suffices to show that $V_X\otimes_{V_1} V_Y$ is isomorphic to $V_X\otimes_{\KQ} V_Y$ as a vector space, therefore making the lax monoidal functor a monoidal functor from $C$ into $\Bim(V_1)$. Consider $J_{X,Y}:F(X)\otimes_{\KQ} F(Y)\rightarrow F(X\otimes Y)$. This is induced by a $\KQ$ balanced map from $F(X)\otimes F(Y)\rightarrow F(X\otimes Y)$. Consider $\leftidx_a x_b\leftidx_b \alpha_c\otimes \leftidx_cy_b$. Since $V_X$ is idempotent split, then it follows that we can braid the path $\alpha$ to get $\sum_v \leftidx_a \beta_v\leftidx_v x_c\otimes \leftidx_cy_b.$ Now since these $V_X$ subspaces cannot contain any edges it follows that this object is not in $V_{X}\otimes_{\KQ}V_Y$ and thus this implies that the relative tensor product is actually understood over $V_1$. This is the same when the edge is multiplied on the $V_Y$ implying that the relative tensor product of base spaces over $\KQ$ is actually the relative tensor product over $V_1$. Thus, concluding that $V_X\otimes_{\KQ}V_Y\cong V_X\otimes_{V_1} V_Y$. Therefore, we can view the image $V_X\otimes_{\KQ}V_Y$ as being over $V_1=\K^{|\KQ_0|}$. This implies that for a separable idempotent split based action, the lax monoidal functor $F:C\rightarrow \Bim(V_1)$ is actually monoidal. Therefore, the lax monoidal functor for a graded idempotent split based action of $\coRep(H)$ on $\KQ$ is, in fact, over $F_H$ the weak fiber functor, thus, it is actually an action of $H$ on $\KQ$
\end{proof}

This lemma allows us to port all of our theory about graded separable idempotent split based action of fusion categories and translate them into actions of weak Hopf algebras on fusion categories.

Now that we have this connection between the category of weak Hopf algebra actions and fusion category actions over $F_H$ we can use the results in \cite{Bet25} to produce a classification of families of graded actions of the weak Hopf algebras on path algebras and then for the weak Hopf algebras corresponding to $\PSU(2)_{p-2}$ we can understand all separable actions in terms of weak Hopf algebra data in some context. 

\begin{lem}
    Let $H$ have a filtration preserving weak Hopf algebra action on $\KQ$, then that action is a grading preserving action of $H$ on $\KQ$.
\end{lem}

\begin{proof}
    The idea of this proof follows from \cite[Proposition 3.19]{EKW21}. If there is a filtered action of $H$ on $\KQ$, then there is $\KQ_0$ that has the structure of an $H$ module algebra. This implies that $\KQ_0\oplus\KQ_1$ has the structure of a $\KQ_0$ bimodule internal to $\Rep(H)$. But since the dual category of a semisimple module category is multifusion, it follows that $\KQ_1$ must also be a bimodule over $\KQ_0,$ implying that the action is graded
\end{proof}

Now we can use our understanding of graded fusion category actions to classify families of filtered actions of weak Hopf algebra actions on path algebras up to conjugacy. As shown in \cite{Bet25}, all separable idempotent split based actions are conjugate to separable idempotent split based actions where the base spaces $V_1$ are the canonical vertex projections of the path algebra. Thus, we produce the following theorem about weak Hopf algebra actions that are strongly isomorphic.

\begin{thm}
    There is a family of graded actions of semisimple weak Hopf algebras $H$ up to $J-$ twisted algebra isomorphism on path algebras classified by the following data,
     \begin{enumerate}
        \item A semisimple $\coRep(H)$ module category structure on $M$ up to invertible $\coRep(H)$ module functors of the form $(\Id,\rho)$.
        \item A conjugacy class of objects $[(Q,\phi_{-})]$ in $\End_{\coRep(H)}(M)$ up to conjugation by 
 a $\coRep(H)$ module endofunctors of the form $(\Id,\rho)$. 
   \end{enumerate} 
\end{thm}
\begin{proof}
    This proof follows from the classification of filtered fusion categories on path algebras \cite[Lemma 5.14]{Bet25}, noting that by Lemma \ref{lem:idem split FH} these actions are weak Hopf algebra actions.
\end{proof}

We currently know very little about other non-filtered weak Hopf algebra actions on path algebras outside of weak Hopf algebras coming from $\PSU(2)_{p-2}$. This fusion category is has the property that its only indecomposable semisimple module category is itself. $\VecC(\Z_p)_\omega$ also has this property , but most fusion categories have non-trivial semisimple module categories. In general, one would need to understand something about separable based actions of fusion categories on path algebras for a specific fusion category to classify the corresponding weak Hopf algebra actions.

We will explicitly write down the weak Hopf algebra action on the path algebra coming from our results on fusion category actions of $\coRep(H)$ on $\KQ$. The weak Hopf Algebra structure from $\coRep(H)$ acting on $M$ can be computed by studying the following vector spaces $M(a,x\triangleright c)\otimes M(x\triangleright b,d)$. These trivalent vector spaces can be represented by the following pictures. This picture, the basis of a weak Hopf algebra, was written down in \cite{BLV} and has been long known to experts.

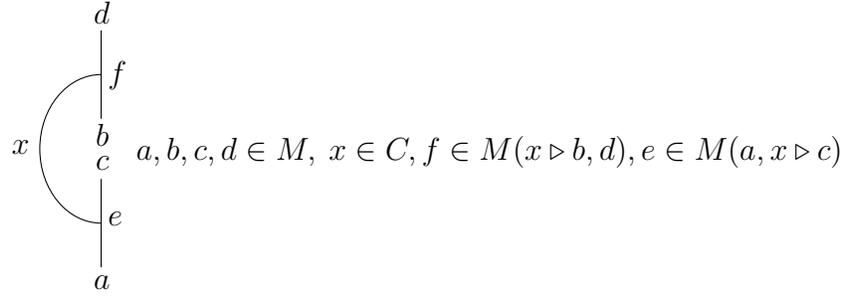
\[\WHA{a}{b}{c}{d}{f}{e}{x} \strut\smash{\raise 1.75cm\hbox{$\hspace{1mm}a,b,c,d \in M,\hspace{1mm} x\in C,f \in M(x\triangleright b,d), e\in M(a,x\triangleright c) $}}\]

Given a graded separable idempotent split based action of $C$ on $\KQ$, we should be able to use our theory to produce a weak Hopf algebra action of the corresponding weak Hopf algebra on $\KQ$.

Let $M$ be a semisimple module category over $C$. Given a $C$-module functor $Q:M \rightarrow M$, we can realize a quiver $Q$ where the vertices are the simple objects and the edges from vertices $X$ to $Y$ are $\Hom(X,Q(Y))$. Building off we can find all paths of length $i$ from $X$ to $Y$ by computing $\dim  (\Hom(X,Q^i(Y)))$.
We can then define the Path Algebra $\KQ$ as 

\[\bigoplus_n \bigoplus_{a,b \in \Irr(C)} \Hom(a,Q^n(b))\]
Pictorially, we can visualize the path algebra as the following picture basis,

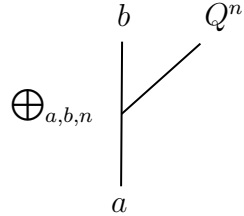
\begin{center}
\tikzset{every picture/.style={line width=0.75pt}} 

\begin{tikzpicture}[x=0.75pt,y=0.75pt,yscale=-1,xscale=1]

\draw    (310,83) -- (309.5,156) ;
\draw    (309.75,119.5) -- (349.5,84) ;

\draw (303,160.4) node [anchor=north west][inner sep=0.75pt]    {$a$};
\draw (306,62.4) node [anchor=north west][inner sep=0.75pt]    {$b$};
\draw (350,62.4) node [anchor=north west][inner sep=0.75pt]    {$Q^{n}$};
\draw (252,105.4) node [anchor=north west][inner sep=0.75pt]    {$\bigoplus _{a,b,n}$};
\end{tikzpicture}

\end{center}

Now, to translate our separable idempotent split graded based action of $C$ into our action of a weak Hopf algebra, as shown earlier, given an $h\in H$ and an $\alpha \in \KQ$, we define our action of $h_{(1)}$ on $a$ by using the half braiding of our bimodule. 

The action is as follows,

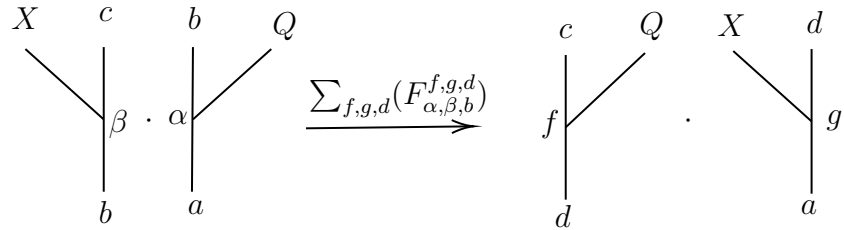
\begin{center}

\tikzset{every picture/.style={line width=0.75pt}} 

\begin{tikzpicture}[x=0.75pt,y=0.75pt,yscale=-1,xscale=1]

\draw    (98,88) -- (98,161) ;
\draw    (98,124.5) -- (58.5,89) ;
\draw    (143,88) -- (142.5,161) ;
\draw    (143,124.5) -- (182.5,89) ;
\draw    (200,129) -- (282.5,128.02) ;
\draw [shift={(284.5,128)}, rotate = 179.32] [color={rgb, 255:red, 0; green, 0; blue, 0 }  ][line width=0.75]    (10.93,-3.29) .. controls (6.95,-1.4) and (3.31,-0.3) .. (0,0) .. controls (3.31,0.3) and (6.95,1.4) .. (10.93,3.29)   ;
\draw    (331,92) -- (331,165) ;
\draw    (331,128.5) -- (371,91) ;
\draw    (455,89) -- (455,162) ;
\draw    (455,125.5) -- (415.5,90) ;

\draw (94,164) node [anchor=north west][inner sep=0.75pt]    {$b$};
\draw (94,66.4) node [anchor=north west][inner sep=0.75pt]    {$c$};
\draw (50,66.4) node [anchor=north west][inner sep=0.75pt]    {$X$};
\draw (117,120) node [anchor=north west][inner sep=0.75pt]    {$\cdot $};
\draw (139,164) node [anchor=north west][inner sep=0.75pt]    {$a$};
\draw (139,66.4) node [anchor=north west][inner sep=0.75pt]    {$b$};
\draw (175,66.4) node [anchor=north west][inner sep=0.75pt]    {$ \begin{array}{l}
Q\\
\end{array}$};
\draw (200,100) node [anchor=north west][inner sep=0.75pt]    {$ {\sum_{f,g,d}
}(F_{\alpha ,\beta ,b}^{f,g,d})$};
\draw (389,120) node [anchor=north west][inner sep=0.75pt]    {$\cdot $};
\draw (324,166) node [anchor=north west][inner sep=0.75pt]    {$d$};
\draw (326,75) node [anchor=north west][inner sep=0.75pt]    {$c$};
\draw (360,66.4) node [anchor=north west][inner sep=0.75pt]    {$ \begin{array}{l}
Q\\
\end{array}$};
\draw (448,164) node [anchor=north west][inner sep=0.75pt]    {$a$};
\draw (451,68.4) node [anchor=north west][inner sep=0.75pt]    {$d$};
\draw (406,70) node [anchor=north west][inner sep=0.75pt]    {$X$};
\draw (129,118.4) node [anchor=north west][inner sep=0.75pt]    {$\alpha $};
\draw (99,118.4) node [anchor=north west][inner sep=0.75pt]    {$\beta $};
\draw (317,118.4) node [anchor=north west][inner sep=0.75pt]    {$f$};
\draw (461,118.4) node [anchor=north west][inner sep=0.75pt]    {$g$};

\end{tikzpicture}
\end{center}

This defines the half braiding of each $\KQ$ bimodule coming from our idempotent split based action of $C$ on $\KQ$. Then we use Lemma \ref{lem:based action give Hopf action} to produce a weak Hopf algebra action on $\KQ$. The action is defined as follows,

\[\strut\smash{\raise 1.78cm\hbox{$\hspace{1mm}\Delta: \hspace{1mm}$}}\WHA{a}{b}{c}{d}{f}{e}{x}\strut\smash{\raise 1.78cm\hbox{$\hspace{1mm}\rightarrow\hspace{1mm}\frac{1}{\sqrt{d_X}}\sum_{m,n,h}$}}\WHA{n}{b}{m}{d}{f}{h}{x} \strut\smash{\raise 1.78cm\hbox{$\hspace{1mm}\otimes \hspace{1mm}$}}\WHA{a}{n}{c}{m}{h}{e}{x}
\] 

Then we define the action of $h_{(1)}$ to be coinside with the half braiding defined above, by Lemma \ref{lem:based action give Hopf action}, this is our weak Hopf algebra action.

\begin{example}
For a more concrete example, let $C\cong \VecC(G)_{\omega}$ for some nontrivial three cocycle $\omega$. Then, since this fusion category does not admit a fiber functor, it is equivalent to $\coRep(H)$ for some weak Hopf algebra $H$. We want to construct the explicit graded action of the face algebra on $\KQ$. That is when our module category is $\VecC(G)_{\omega}$. Let's obeserve when $X$ and $Q$ are simple objects $\delta_a$ and $\delta_b$. Here, the morphisms will be suppressed since we can choose our morphism basis to be just the identity morphism.

Given, \begin{center}

\tikzset{every picture/.style={line width=0.75pt}} 

\begin{tikzpicture}[x=0.75pt,y=0.75pt,yscale=-1,xscale=1]

\draw    (98,88) -- (98,161) ;
\draw    (98,124.5) -- (58.5,89) ;
\draw    (143,88) -- (142.5,161) ;
\draw    (143,124.5) -- (182.5,89) ;
\draw    (200,129) -- (282.5,128.02) ;
\draw [shift={(284.5,128)}, rotate = 179.32] [color={rgb, 255:red, 0; green, 0; blue, 0 }  ][line width=0.75]    (10.93,-3.29) .. controls (6.95,-1.4) and (3.31,-0.3) .. (0,0) .. controls (3.31,0.3) and (6.95,1.4) .. (10.93,3.29)   ;
\draw    (331,92) -- (331,165) ;
\draw    (331,128.5) -- (371,91) ;
\draw    (455,89) -- (455,162) ;
\draw    (455,125.5) -- (415.5,90) ;

\draw (94,164) node [anchor=north west][inner sep=0.75pt]    {$\delta_{kb^{-1}}$};
\draw (94,66.4) node [anchor=north west][inner sep=0.75pt]    {$\delta_h$};
\draw (50,66.4) node [anchor=north west][inner sep=0.75pt]    {$\delta_a$};
\draw (117,120) node [anchor=north west][inner sep=0.75pt]    {$\cdot $};
\draw (139,164) node [anchor=north west][inner sep=0.75pt]    {$\delta_k$};
\draw (139,66.4) node [anchor=north west][inner sep=0.75pt]    {$\delta_{kb^{-1}}$};
\draw (175,66.4) node [anchor=north west][inner sep=0.75pt]    {$ \begin{array}{l}
\delta_b\\
\end{array}$};
\draw (200,100) node [anchor=north west][inner sep=0.75pt]    {$ \hspace{5mm}\omega_{a,h,b}$};
\draw (389,120) node [anchor=north west][inner sep=0.75pt]    {$\cdot $};
\draw (324,166) node [anchor=north west][inner sep=0.75pt]    {$\delta_{a^{-1}k}$};
\draw (326,75) node [anchor=north west][inner sep=0.75pt]    {$\delta_h$};
\draw (360,66.4) node [anchor=north west][inner sep=0.75pt]    {$ \begin{array}{l}
\delta_b\\
\end{array}$};
\draw (448,164) node [anchor=north west][inner sep=0.75pt]    {$\delta_k$};
\draw (451,68.4) node [anchor=north west][inner sep=0.75pt]    {$\delta_{a^{-1}k}$};
\draw (406,70) node [anchor=north west][inner sep=0.75pt]    {$\delta_a$};
\draw (129,110) node [anchor=north west][inner sep=0.75pt]    {};
\draw (99,112) node [anchor=north west][inner sep=0.75pt]    {};
\draw (317,118.4) node [anchor=north west][inner sep=0.75pt]    {};
\draw (461,115.4) node [anchor=north west][inner sep=0.75pt]    {};

\end{tikzpicture}
\end{center}

Now computing the following comultiplication,

\[\strut\smash{\raise 1.78cm\hbox{$\hspace{1mm}\Delta: \hspace{1mm}$}}\WHA{
\delta_a}{{\delta_{g^{-1}d}}}{{\delta_{g^{-1}a}}}{\delta_d}{}{}{\delta_g}\strut\smash{\raise 1.78cm\hbox{$\hspace{1mm}\rightarrow\hspace{1mm}\sum_{n}$}}\WHA{\delta_n}{\delta_{g^{-1}d}}{\delta_{g^{-1}n}}{\delta_d}{}{}{\delta_g} \strut\smash{\raise 1.78cm\hbox{$\hspace{1mm}\otimes \hspace{1mm}$}}\WHA{\delta_a}{\delta_n}{\delta_{g^{-1}a}}{\delta_{g^{-1}n}}{}{}{\delta_g}
.\] 
\end{example}

Now the action here can be described by the following,
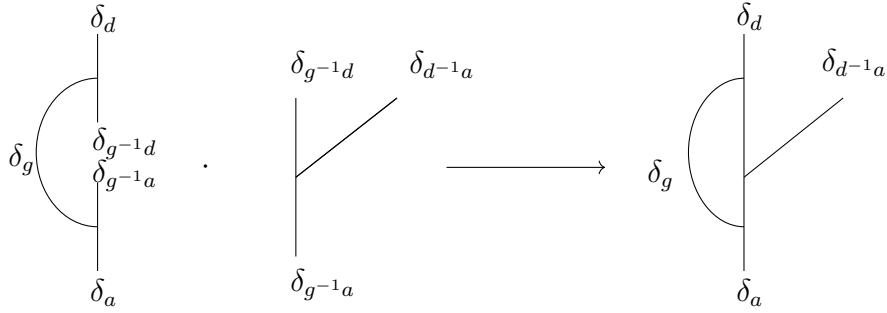
\begin{center}
\begin{tikzpicture}[x=0.75pt,y=0.75pt,yscale=-1,xscale=1]

\draw    (223.99,132.68) -- (223.99,177.28) ;
\draw    (223.99,57.69) -- (223.99,102.29) ;
\draw  [draw opacity=0] (224,154.98) .. controls (206.91,154.97) and (193.05,138.18) .. (193.05,117.48) .. controls (193.05,96.78) and (206.91,80) .. (224,79.99) -- (224.02,117.48) -- cycle ; \draw   (224,154.98) .. controls (206.91,154.97) and (193.05,138.18) .. (193.05,117.48) .. controls (193.05,96.78) and (206.91,80) .. (224,79.99) ;  

\draw (177.73,112) node [anchor=north west][inner sep=0.75pt]    {$\delta_g$};
\draw (219.03,180) node [anchor=north west][inner sep=0.75pt]    {$\delta_a$};
\draw (226.25,148) node [anchor=north west][inner sep=0.75pt]    {};
\draw (220,120) node [anchor=north west][inner sep=0.75pt]  {$\delta_{g^{-1}a}$};
\draw (220,103.64) node [anchor=north west][inner sep=0.75pt]    {$\delta_{g^{-1}d}$};
\draw (218.99,42) node [anchor=north west][inner sep=0.75pt]    {$\delta_d$};
\draw (275,120) node [anchor=north west][inner sep=0.75pt]    {$\cdot$};

\draw    (323.99,170) -- (323.99,90) ;
\draw   (323.99,130) -- (375,90) ;
\draw (320,65) node [anchor=north west][inner sep=0.75pt]    {$\delta_{g^{-1}d}$};
\draw (320,175) node [anchor=north west][inner sep=0.75pt]    {$\delta_{g^{-1}a}$};
\draw (380,65) node [anchor=north west][inner sep=0.75pt]    {$\delta_{d^{-1}a}$};

\draw[<-]   (480,125) -- (400,125) ;

\draw  [draw opacity=0] (550,154.98) .. controls (535,154.97) and (522,138.18) .. (522,117.48) .. controls (522,96.78) and (535,80) .. (535,79.99) -- (550,117.48) -- cycle ; \draw   (550,154.98) .. controls (535,154.97) and (522,138.18) .. (522,117.48) .. controls (522,96.78) and (535,80) .. (550,79.99) ; 

\draw    (550,57.69) -- (550,177.28) ;
\draw   (550,130) -- (600,90) ;
\draw   (323.99,130) -- (375,90) ;

\draw (545,41) node [anchor=north west][inner sep=0.75pt]    {$\delta_d$};
\draw (545,180) node [anchor=north west][inner sep=0.75pt]    {$\delta_a$};

\draw (588,64) node [anchor=north west][inner sep=0.75pt]    {$\delta_{d^{-1}a}$};
\draw (500,120) node [anchor=north west][inner sep=0.75pt]    {$\delta_g$};
\end{tikzpicture}
\end{center}

Which, when applying the $F$ symbol, produces the following picture representing an element in the path algebra. This construction will work given any semisimple module category and a progenerator of $\VecC(G)_\omega$.

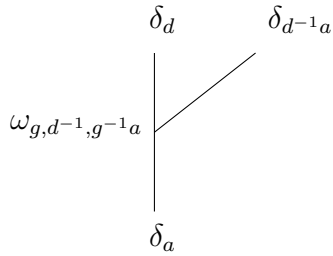
\begin{center}

\begin{tikzpicture}[x=0.75pt,y=0.75pt,yscale=-1,xscale=1]

\draw    (323.99,170) -- (323.99,90) ;
\draw   (323.99,130) -- (375,90) ;
\draw (250,120) node [anchor=north west][inner sep=0.75pt]    {$\omega_{g,d^{-1},g^{-1}a}$};
\draw (320,65) node [anchor=north west][inner sep=0.75pt]    {$\delta_{d}$};
\draw (320,175) node [anchor=north west][inner sep=0.75pt]    {$\delta_{a}$};
\draw (380,65) node [anchor=north west][inner sep=0.75pt]    {$\delta_{d^{-1}a}$};

\end{tikzpicture}
\end{center}
\begin{example}
The Fibonacci Category $\Fib$ is, the category with $2$ simple objects $1$ and $\tau$ such that $\tau\otimes \tau\cong 1\oplus \tau$. Given a fusion category $\Fib$ and a semisimple module category $M$ and a progenerator which is the direct sum of simple objects in $M$, we can produce weak Hopf algebra. When $M\cong \Fib$, this is the canonical face algebra associated to $\Fib$. The basis for the face algebra associated with $\Fib$ can be written as follows. It is a $13$ dimensional algebra spanned by the following diagrams. The basis of morphisms for the category consists of $\alpha_1:1\rightarrow 1$ and $\alpha_{\tau}:\tau\rightarrow \tau$.

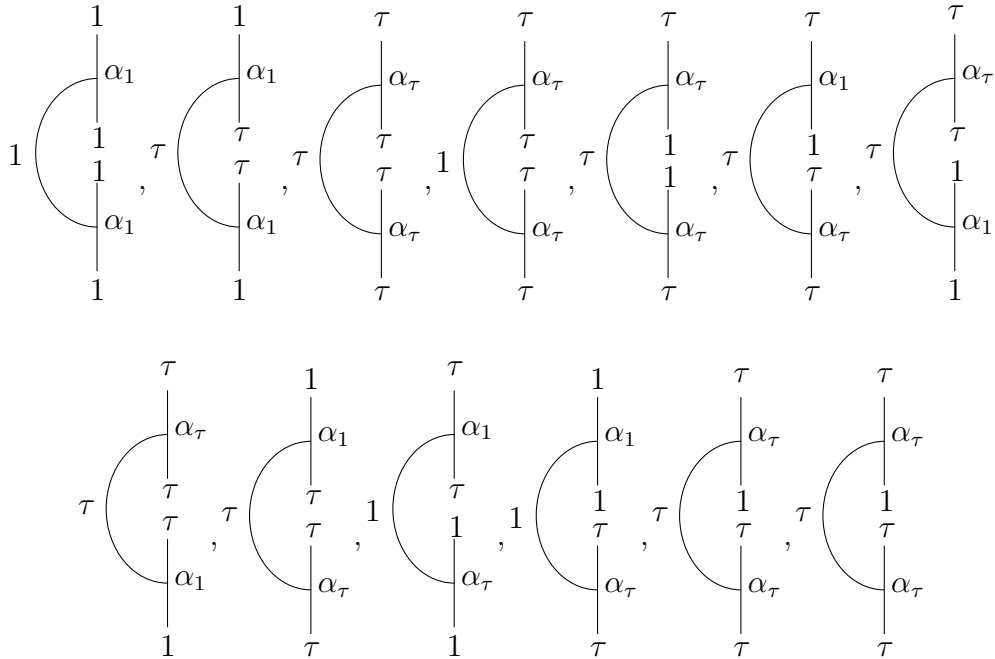
\[\WHA{1}{1}{1}{1}{\alpha_1}{\alpha_1}{1} \strut\smash{\raise 1.5cm\hbox{,}}\WHA{1}{\tau}{\tau}{1}{\alpha_1}{\alpha_1}{\tau} \strut\smash{\raise 1.5cm\hbox{,}}\WHA{\tau}{\tau}{\tau}{\tau}{\alpha_\tau}{\alpha_\tau}{\tau} \strut\smash{\raise 1.5cm\hbox{,}}\WHA{\tau}{\tau}{\tau}{\tau}{\alpha_\tau}{\alpha_\tau}{1} \strut\smash{\raise 1.5cm\hbox{,}}\WHA{\tau}{1}{1}{\tau}{\alpha_\tau}{\alpha_\tau}{\tau} \strut\smash{\raise 1.5cm\hbox{,}}\WHA{\tau}{1}{\tau}{\tau}{\alpha_1}{\alpha_\tau}{\tau}\strut\smash{\raise 1.5cm\hbox{,}}\WHA{1}{\tau}{1}{\tau}{\alpha_\tau}{\alpha_1}{\tau}\]

\[\WHA{1}{\tau}{\tau}{\tau}{\alpha_\tau}{\alpha_1}{\tau} \strut\smash{\raise 1.5cm\hbox{,}}\WHA{\tau}{\tau}{\tau}{1}{\alpha_1}{\alpha_\tau}{\tau} \strut\smash{\raise 1.5cm\hbox{,}}\WHA{1}{\tau}{1}{\tau}{\alpha_1}{\alpha_\tau}{1} \strut\smash{\raise 1.5cm\hbox{,}}\WHA{\tau}{1}{\tau}{1}{\alpha_1}{\alpha_\tau}{1} \strut\smash{\raise 1.5cm\hbox{,}}\WHA{\tau}{1}{\tau}{\tau}{\alpha_\tau}{\alpha_\tau}{\tau} \strut\smash{\raise 1.5cm\hbox{,}}\WHA{\tau}{1}{\tau}{\tau}{\alpha_\tau}{\alpha_\tau}{\tau} \]

Then the multiplication, comultiplication, unit, counit, and antipode can be computed using the formulas in the appendix of \cite{BLV}. If there were an action of this weak Hopf algebra on a path algebra by the result above, it would preserve the grading of the path algebra.

 By \cite{BD12} up to monoidal equivalence the associator is trivial except for $(\tau\otimes \tau)\otimes \tau \xrightarrow[]{\alpha}\tau \otimes (\tau\otimes \tau)$. In this scenario, it follows that the $F$ symbols or $6j$ symbols can be described by the following complex number and $2\times 2$ matrix

\[\alpha=1,\hspace{1mm} A=\begin{bmatrix}
a & 1 \\
-a & -a 
\end{bmatrix}.\]

Here $a$ is a root of $x^2-x-1$. We can explicitly compute the action of the face algebra on $\KQ$ where $Q$ is constructed by an object in $\Fib$. 

\end{example}

Now it was shown in \cite{Bet25} that for the fusion category $\PSU(2)_{p-2}$, these graded fusion category actions are the only separable idempotent split based actions on the path algebra. This allows us to produce the following theorem,

\begin{thm} \label{thm:classification PSU}
All separable idempotent split based actions of $\PSU(2)_{p-2}$ up on $\KQ$ up to conjugacy are weak Hopf algebra actions on $\KQ$.
\end{thm}
    
\begin{proof}
    This follows from \cite[Theorem 6.17]{Bet25} and the fact that all separable idempotent split based actions up to conjugacy are graded.
\end{proof}

\begin{remark}
    When studying the lax monoidal functor $G:\PSU(2)_{p-2}\rightarrow \VecC$. When the module category is $\PSU(2)_{p-2}^{\oplus n}$, $n\geq 2$, these actions are not weak Hopf algebra actions corresponding to the face algebra associated with a fusion category, but instead are weak Hopf algebra actions coming from the module category $\PSU(2)_{p-2}^{\oplus n}$ and the progenerator that is the direct sum of all simple objects. These weak Hopf algebra actions can still be understood in terms of data coming from our face algebra actions on path algebras since in each partition the half braiding can be expressed in terms of the $6j$ symbols for $\PSU(2)_{p-2}$ and thus can be expressed in terms of an data from an action of the face algebra on $\KQ$.
\end{remark}

Now \cite[Theorem 3.24]{Bet25} implies that any separable based action induces an idempotent split based action canonically. This process doesn't change the half braiding but alters the $V_X$ spaces by including more orthogonal idempotent projections. 

\begin{lem}
    Consider a separable based action of $\PSU(2)_{p-2}$ on $\KQ$, then this action can be understood in terms of a weak Hopf algebra action of $H$ on $\KQ$
\end{lem}
\begin{proof}
    Let $(F,J,V)$ be a separable based action of $\PSU(2)_{p-2}$  on $\KQ$. Then this induces an idempotent split based action of $\PSU(2)_{p-2}$ on $\KQ$. These actions are understood to be weak Hopf algebra actions, and thus the result follows.
\end{proof}

This lemma states that every separable idempotent split based action of $\PSU(2)_{p-2}$ can be understood in the right context as a weak Hopf algebra action on $\KQ$. Given an idempotent split action we can recover a separable action when there is a unital subalgebra $H\leq G$. 

\section*{Acknowledgements}
The author would like to thank Corey Jones. His knowledge, advice, and words of encouragement throughout made this possible. This work was partially supported by NSF DMS 2247202.

\bibliographystyle{alpha}
\bibliography{main}

\end{document}